\documentclass[11pt,reqno]{amsart}
\usepackage{amssymb,hyperref,enumerate}
\usepackage[abbrev]{amsrefs}

\numberwithin{equation}{section}

\newtheorem{theorem}{Theorem}[section]
\newtheorem{lemma}[theorem]{Lemma}
\newtheorem{proposition}[theorem]{Proposition}
\newtheorem{definition}[theorem]{Definition}
\newtheorem{remark}[theorem]{Remark}
\newtheorem{corollary}[theorem]{Corollary}

\def\ON{\operatorname}

\def\sp{{\rm Supp}}
\def\c{{\rm cap}}
\def\d{{\rm dist}}

\author[Sunghyun Hong]{}
\email{shhong7523@gmail.com}

\begin{document}
\title[Symplectic capacity for the Zakharov system]{Infinite dimensional symplectic capacity and nonsqueezing property for the Zakharov system on the $1$-dimensional torus}
\maketitle

\centerline{\scshape Sunghyun Hong\footnote{Email address: shhong7523@gmail.com}}
\medskip
\medskip

\begin{abstract}
We prove the invariant of the symplectic capacity for the Zakharov system on a torus. If the Zakharov solution map is well-defined, then it can be regarded as a symplectomorphism. Thus, we first show the global well-posedness via the local well-posedness and the conservation law. The invariant of the symplectic capacity can be obtained using an approximation method. Many authors   use an approximation method to obtain the nonsqueezing theorem, instead of an invariant of the symplectic capacity. However, the conditions of the Hamiltonian system introduced by Kuksin can be relaxed by a new modified infinite dimensional Hamiltonian system. Thus we can back to the symplectic capacity which contains the nonsqueezing property. Heuristically, we obtain the invariant by using the Hamiltonian system which has linear flow at high frequencies and nonlinear flow at low frequencies. 
\end{abstract}

\tableofcontents

\section{Introduction}
In this paper, we consider the Zakharov system
\begin{equation}\label{eq:Zakharov}
\left\{\begin{array}{ll}
i \partial_t u + \alpha \partial_x^2 u =un, \hspace{1em} &\left(t,x\right) \in \mathbb{R} \times \mathbb{T}, \\
\beta^{-2} \partial_t^2 n - \partial_x^2 n = \partial_x^2 \left(\left|u\right|^2\right), &\left(t,x\right) \in \mathbb{R} \times \mathbb{T}, \\
\left.\left(u,n, \partial_t n\right)\right|_{t=0} = \left(u_0(x),n_0(x), n_1(x)\right) \in L_x^2 \times H_x^{-1/2} \times H_x^{-3/2}, & x \in \mathbb{T}\left(:=\mathbb{R}/{2\pi \mathbb{Z}}\right),
\end{array}
\right.
\end{equation}
where $\alpha$ and $\beta$ are positive real constants, respectively. The functions $u$ and $n$ are complex valued and real valued, respectively. The Zakharov system (\ref{eq:Zakharov}) enjoys conservation laws,
\begin{equation}\label{eq:mass conservation u}
M\left[u\right]\left(t\right) = \int_{\mathbb{T}} \left|u\left(t\right)\right|^2 dx  =  \int_{\mathbb{T}} \left|u\left(0\right)\right|^2 dx = M\left[u_0\right],
\end{equation}
and
\begin{equation}\label{eq:Hamiltonian}
\begin{aligned}
&H\left[u,n,\partial_t n\right]\left(t\right) := H\left[u,n, \dot n\right]\left(t\right)\\
&= \int_{\mathbb{T}} \left(\alpha\left|\partial_x u\left(t\right)\right|^2 + \frac{\left|n(t)\right|^2}{2}  + \frac{\beta^2\left|i \partial_x^{-1} \dot n\right|^2}{2}+ n\left(t\right)\left|u\left(t\right)\right|^2\right) dx = H\left[u_0,n_0,n_1\right].
\end{aligned}
\end{equation}
The first (\ref{eq:mass conservation u}) is called the mass conservation law and the second (\ref{eq:Hamiltonian}) is called the Hamiltonian. They are important tools for showing global well-posedness and to define the symplectic capacity, respectively. 
From \eqref{eq:Zakharov}, we have
\begin{equation*}
\frac{d^2}{dt^2}\int_{\mathbb{T}} n\left(t,x\right) dx  =  \beta^2 \int_{\mathbb{T}} \partial^2_x \left(n+\left|u\right|^2\right) dx =0,
\end{equation*}
and so
\begin{equation*}
\int_{\mathbb{T}} n\left(t,x\right) dx =c_1t +c_0,
\end{equation*}
where
\begin{equation*}
c_0 = \int_{\mathbb{T}} n_0\left(x\right) dx~ \text{and}~ c_1  =  \int_{\mathbb{T}} n_1\left(x\right) dx.
\end{equation*}
Hence, we denote
\begin{equation}\label{eq: Gal trans}
u'\left(t,x\right) = \exp\left[i\left(\frac{1}{4\pi} c_1 t^2 + c_0 t\right)\right] u\left(t,x\right) ~ \text{and}~ n'\left(t,x\right) = n\left(t,x\right) - \frac{1}{2\pi}\left(c_1 t + c_0\right),
\end{equation}
then $u'$ and $n'$ are also the solutions to (\ref{eq:Zakharov}), and
\begin{equation*}
\int n'\left(t,x\right)dx =0 ~ \text{and} ~\int \partial_t n'\left(t,x\right)dx =0.
\end{equation*}
If initial data $n_0\left(x\right)$, $n_1\left(x\right)$ have general mean, then one can easily change the data into the mean zero data by \eqref{eq: Gal trans}. Therefore, it will be convenient to work in the case when initial data $n_0\left(x\right)$, $n_1\left(x\right)$ have mean zero. 

The system (\ref{eq:Zakharov}) was introduced by Zakharov \cite{Zakharov:1972tz}. It represents the propagation of Langmuir turbulence waves in unmagnetized ionized plasma \cite{Zakharov:1972tz}. In the system, $u\left(t,x\right)$ expresses the slowly varying envelope of the electric field and $n\left(t,x\right)$ describes the deviation in ion density from its mean. The constant $\alpha$ is a dispersion coefficient and the constant $\beta$ is the speed of an ion acoustic wave in plasma. 

There are many results for the symplectic capacity and the nonsqueezing theorem for the infinite dimensional Hamiltonian system.  The symplectic capacity was introduced by Ekeland and Hofer \cite{Ekeland:1989dh,Ekeland:1990is} for $\mathbb{R}^{2n}$, and by Hofer and Zehnder \cite{Hofer:1990ul,Hofer:2011vo} for $2n$-dimensional general symplectic manifolds. It was developed from the Darboux width, which was discovered by Gromov \cite{Gromov:1985ww}. Kuksin \cite{Kuksin:1995ue} was the first contributor of the infinite dimensional symplectic capacity for Hamiltonian Partial Differential Equations(PDEs).
Kuksin's concept, of course, is based on the finite dimensional symplectic capacity which was developed by Hofer and Zehnder. Indeed, Kuksin proved an invariance in the symplectic capacity for particular Hamiltonian flow, and so he also captured its nonsqueezing property. Furthermore, he introduced an abstract method in which the Hamiltonian flow on the appropriate function space can be regarded as a symplectic map. Although there are results which have applied this condition \cite{Kuksin:1995ue,Roumegoux:2010sn}, Kuksin's condition for solution flow is somewhat strong. Thus, many contributors to this issue have turned to the nonsqueezing theorem for specific equations. 
\\

To prove the nonsqueezing results for Haimtonian PDEs, one of the main steps is to find a `good' truncation. Besides, the given Hamiltonian system turns out to be well-behaved with `good' frequency truncations. There are two techniques for the truncation, the methods of \cite{Bourgain:1994tr} and \cite{Colliander:2005vv}. In \cite{Bourgain:1994tr}, Bourgain proved the nonsqueezing theorem of the cubic nonlinear Schr\"odinger equation (NLS) in its phase space $L^2_x\left(\mathbb{T}\right)$ space. A sharp frequency truncation and the $X^{s,b}$ space were used to approximate the original solution. Later, this argument was extended by Colliander et al. \cite{Colliander:2005vv} for the KdV equation in its phase space $H^{-1/2}_x\left(\mathbb{T}\right)$. The argument in \cite{Colliander:2005vv} is more complex than the one in \cite{Bourgain:1994tr}. They used a smooth truncation, and also used the Miura transform which changes the KdV flow to a mKdV flow. Indeed, they showed an approximation using truncated mKdV flow and used Miura transform and its inverse. In this way, they obtained the estimate for the KdV flow. We use the methods of Bourgain \cite{Bourgain:1994tr} instead of the method of Colliander et al. \cite{Colliander:2005vv}, because the modulation effects from the non-resonant interaction of (\ref{eq:Zakharov}) is better than that of the KdV equation. In Section \ref{sec:estimates}, we show the bilinear estimates produced by these modulation effects and a similar calculation in \cite{Colliander:2008cq,Takaoka:1999uw}. Specifically, bilinear estimates are needed to approximate the truncated solution flow, and this is stronger than the estimates of \cite{Takaoka:1999uw} to prove the local well-posedness.
Hong and Kwak \cite{Hong:2016fn} extended the result to the higher-order KdV equation, and Mendelson \cite{Mendelson:2014vh} also showed the nonsqueezing of the Klein-Gordon flow on $\mathbb{T}^3$ via a probabilistic approach. Moreover, Kwak \cite{Kwak:2017wb} proved the nonsqueezing and the local well-posedness for the fourth-order cubic nonlinear Schr\"odinger equation on a torus. Recently, Killip et al. \cite{Killip:2016vd,Killip:2016wj} proved the nonsqueezing theorem of the cubic NLS equation on a real line and a plane, respectively. These results are the first nonsqueezing study for an unbounded domain. 

Nevertheless, we want to go back to the `capacity' beyond `nonsqueezing.' There are some results that are independent of the nonsqueezing theorem. For example,  Abbondandolo and Majer \cite{Abbondandolo:2015cb} constructed the symplectic capacity on a \emph{convex set} in the Hilbert space without the approximation approach. However, we focus on the relaxation of Kuksin's condition. As a result, we obtain a symplectic capacity for the Zakharov system flow which does not satisfy Kuksin's condition. In particular, there is no nonsqueezing result associated with the Zakharov flow. Moreover, we do not even know the global well-posedness for the symplectic Hilbert space $L_x^2\left(\mathbb{T}\right) \times H_x^{-\frac{1}{2}}\left(\mathbb{T}\right) \times H_x^{-\frac{3}{2}}\left(\mathbb{T}\right)$. We use the appropriate frequency truncation and approximate the finite dimensional solution to the original infinite dimensional solution, preserving the symplectic form. These are nontrivial facts, because the nonlinear terms in the Zakharov system does not satisfy Kuksin's results. To overcome these obstacles, we need to prove that the frequency truncated solution flow well-approximates to the original solution flow. In addition, the truncated flow should be a Hamiltonian flow. We now introduce the main result.

\begin{theorem}\label{thm: main thm}
Assume that $\frac{\beta}{\alpha}$ is not an integer. Let $Z\left(T\right)$ be the Zakharov flow map at time $T$. For any bounded domain $\mathcal{O}$ in $L_x^2 \times H_x^{-\frac{1}{2}} \times H^{-\frac{3}{2}}_x$, we have
\begin{equation*}
\c\left(\mathcal{O}\right) = \c\left(Z\left(T\right)\left(\mathcal{O}\right)\right)
\end{equation*}
where $\c\left(\cdot\right)$ is the infinite dimensional symplectic capacity. 
\end{theorem}

For the solution map to exist as the symplectic map for any $T>0$, we should have the global well-posedness in the phase space as follows.

\begin{theorem}\label{thm:GWP}
Assume that $\frac{\beta}{\alpha}$ is not an integer. The initial value problem \eqref{eq:Zakharov} is globally well-posed for any $\left(u_0, n_0, n_1\right) \in L_x^2\left(\mathbb{T}\right) \times H_x^{-\frac{1}{2}}\left(\mathbb{T}\right) \times H^{-\frac{3}{2}}_x\left(\mathbb{T}\right)$.
\end{theorem}
Theorem \ref{thm:GWP} can be proved by combining the local well-posedness with the mass conservation of $u$ (\ref{eq:mass conservation u}). The details are in Section \ref{sec:GWP}. It is the Duhamel's formula for (\ref{eq:Zakharov}) which can be written as follows,
\begin{align}
u\left(t\right) &:= S\left(t\right)\left(u_0,n_0,n_1\right) = U\left(t\right) u_0 - i \int_0^t U\left(t-s\right) \left[un\right] \left(s\right) ds, \label{eq:Duhamel_Schrodinger}\\
n\left(t\right) &:=W\left(t\right)\left(u_0,n_0,n_1\right)= \partial_t V\left(t\right) n_0 + V\left(t\right) n_1 + \beta^2 \int_0^t V\left(t-s\right) \partial_x^2 \left[\left|u\right|^2\right]\left(s\right) ds, \label{eq:Duhamel_wave}
\end{align}
where $U\left(t\right) = e^{i\alpha t \partial^2_x}$ and $V\left(t\right) = \frac{\sin \left(\beta t \left(- \partial_x^2\right)^{1/2}\right)}{\beta \left(- \partial_x^2\right)^{1/2}} = \frac{\sin \left(\beta t \sqrt{- \Delta}\right)}{\beta \sqrt{- \Delta}}$. 
We denote the solution to \eqref{eq:Zakharov} by
\begin{equation*}
\begin{aligned}
{\bf{z}}\left(t,x\right) &= \left(u\left(t,x\right),n\left(t,x\right), \partial_t n\left(t,x\right)\right) \\
&:=Z\left(t\right)\left(u_0(x),n_0(x),n_1(x)\right) \\
&= S\left(t\right)\left(u_0,n_0,n_1\right) \times W\left(t\right)\left(u_0,n_0,n_1\right) \times \partial_t W\left(t\right)\left(u_0,n_0,n_1\right).
\end{aligned}
\end{equation*}
Thus, we also have $Z\left(t\right)$ as the solution flow to (\ref{eq:Zakharov}). In the same way as here, we will use bold fonts to present vectors in the appropriate space.  The spatial Sobolev space is given by
\begin{equation*}\label{eq:H^s space}
\left\|u\right\|_{H_x^s} = \left\|\left<k\right>^s \widehat{u} \right\|_{\ell^2_k} := \frac{1}{\left(2\pi\right)^{1/2}} \left(\sum_{k \in \mathbb{Z}} \left<k\right>^{2s}\left|\widehat{u} \right|^2\right)^{1/2}
\end{equation*}
for $s \in \mathbb{R}$, where $\left<k\right> = \left(1+\left|k\right|^2\right)^{1/2}$. Let $\mathcal{H}$ be the symplectic Hilbert space $L_x^2\left(\mathbb{T}\right) \times H_x^{-\frac{1}{2}}\left(\mathbb{T}\right) \times H^{-\frac{3}{2}}_x\left(\mathbb{T}\right)$, and
\begin{equation*}
\left\|\left(u,v,w\right)\right\|_{\mathcal{H}} = \left\|u\right\|_{L^2_x} + \left\|v\right\|_{H^{-1/2}_x} + \left\|w\right\|_{H^{-3/2}_x}.
\end{equation*}
We also define the absolute value in $\mathcal{H}$ by
\begin{equation*}
\left|\left(\widehat{u}_{k_0},\widehat{v}_{k_0},\widehat{w}_{k_0}\right)\right|_{\mathcal{H}} = \left|\widehat{u} _{k_0}\right| + \left|{k_0}\right|^{-\frac{1}{2}}\left|\widehat{v}_{k_0}\right| + \left|{k_0}\right|^{-\frac{3}{2}}\left|\widehat{w}_{k_0}\right|
\end{equation*}
for fixed frequency component $k_0$. 

From Theorem \ref{thm: main thm}, we can consider the nonsqueezing theorem of the Zakharov system as well. We first define a ball and a cylinder in the function space $\mathcal{H}$.
\begin{definition}
Let $B^{\infty}_R\left({\bf{v}}_*\right)$ be an infinite dimensional ball in ${\mathcal{H}}$ which has the radius $R$ and  is centered at ${\bf{v}}_* \in {\mathcal{H}}$. That is,
\begin{equation*}
B^{\infty}_R\left({\bf{v}}_*\right) := \left\{{\bf{v}} \in {\mathcal{H}} : \left\|{\bf{v}}-{\bf{v}}_*\right\|_{\mathcal{H}} \leq R\right\}.
\end{equation*}
For any $k \in \mathbb{Z}\setminus \left\{0\right\}\left(:=\mathbb{Z}^*\right)$, $C^{\infty}_{k,r}\left(\eta \right)$ is defined an infinite dimensional $k$-th cylinder in ${\mathcal{H}}$ which has the radius $r$ and is centered at $\eta  \in \mathbb{C}^3$. That is,
\begin{equation*}
C^{\infty}_{k,r}\left(\eta \right) := \left\{{\bf{v}} \in {\mathcal{H}} : \left|\widehat{{\bf{v}}} _k-\eta \right|_{\mathcal{H}} \leq r\right\}
\end{equation*}
where $\widehat{\bf{v}}_k = \left(\widehat{u}_k,\widehat{v}_k,\widehat{w}_k\right) \in \mathbb{C}^3$.
\end{definition}
The nonsqueezing property of the Zakharov system is as follows,
\begin{corollary}\label{cor:main nonsq}
Let $ 0 < r< R$, $\left(u^*, n^*,n^{**}\right) \in L_x^2 \times H_x^{-\frac{1}{2}} \times H^{-\frac{3}{2}}_x $, $k \in \mathbb{Z}^*$, $\left(z,w_0,w_1\right) \in \mathbb{C}^3$ and $T>0$. Then
\begin{equation*}
Z \left(T\right) \left( B^{\infty}_{R} \left(u^*,n^*,n^{**}\right)\right) \not \subseteq C^{\infty}_{k,r}\left(z,w_0,w_1\right).
\end{equation*}
In other words, there is a solution $Z\left(T\right)\left(u_0,n_0,n_1\right) \in C_tL_x^2 \times C_tH_x^{-\frac{1}{2}}\times C_tH_x^{-\frac{3}{2}}$ to \eqref{eq:Zakharov} and $ k_0 \in \mathbb{Z^*} $ such that
\begin{equation*}
\left\|\left(u_0,n_0,n_1\right)-\left(u^*,n^*,n^{**}\right)\right\|_{\mathcal{H}} \leq R,
\end{equation*}
and
\begin{equation*}
\left|\mathcal{F}_x\left[\left(Z\left(T\right)\left(u_0,n_0,n_1\right)\right)\right]\left(k_0\right)-\left(z,w_0,w_1\right)\right|_{\mathcal{H}} > r
\end{equation*}
where $\mathcal{F}_x\left[\cdot\right]$ is the spatial Fourier transform.
\end{corollary}
\begin{remark}
There are no smallness conditions imposed on $\left(u^*, n^*, n^{**}\right)$, $\left(z,w_0,w_1\right)$, ${R}$ or ${T}$ in Corollary \ref{cor:main nonsq}.
\end{remark}
Corollary \ref{cor:main nonsq}, the symplectic nonsqueezing theorem, tells us the Zakharov flow cannot squash a large ball into a narrow cylinder, despite the fact that the cylinder has infinite volume.
\section{Notations and Function spaces}
In this section, we introduce notations to discuss our argument. The spatial Fourier transform, the space-time Fourier transform and the inverse Fourier transform are defined as follows,
\begin{equation*}
\begin{aligned}
&\mathcal{F}\left(u\right) = \tilde u\left(k,\tau\right)  = \iint_{\mathbb{R} \times \mathbb{T}} e^{-ikx}e^{-i\tau t} u\left(t,x\right) dxdt, \\
&\mathcal{F}\left(u,v,w\right) = \left(\tilde u , \tilde v, \tilde w\right),\\
&\mathcal{F}_x \left(u\right) =\widehat{u} _k =  \int _{\mathbb{T}} e^{-ikx}u\left(x\right) dx, \\
&\mathcal{F}_x\left(u,v,w\right) =  \left(\widehat{u}_k , \widehat{v}_k , \widehat{w}_k \right),\\
&u \left(x \right) =  \int  e^{ikx}\widehat{u} \left(k\right) dk:=\frac{1}{2 \pi}\sum_{k \in \mathbb{Z}} \widehat{u} _k e^{ikx}.
\end{aligned}
\end{equation*}
We also introduce function spaces. First of all, we use the $X^{s,b}$ spaces (the Bourgain spaces) which are defined by the following norms,
\begin{align*}
\left\|f\right\|_{X_S^{s,b}} &= \left\|\left<k\right>^s \left<\tau - \alpha k^2 \right>^b \tilde f\left(k, \tau\right)\right\|_{l^2_k L^2_{\tau}}, \\
\left\|f\right\|_{X_W^{s,b}} &= \left\|\left<k\right>^s \left<\left|\tau\right| - \beta \left|k\right| \right>^b \tilde f\left(k, \tau\right)\right\|_{l^2_k L^2_{\tau}}.
\end{align*}
Note that the first and the second are associated with Schr\"odinger flow and wave flow, respectively. Using the $X^{s,b}$ spaces, we define $Y_S^s$, $Z_S^s$, $Y_W^s$ and $Z_W^s$ spaces for the solution and the nonlinear terms,
\begin{align*}
\left\|f\right\|_{Y_S^{s}} &=\left\|f\right\|_{X^{s,1/2}_S}+ \left\|\left<k\right>^s  \tilde f\left(k, \tau\right)\right\|_{l^2_k L^1_{\tau}}, \\
\left\|f\right\|_{Z_S^{s}} &=\left\|f\right\|_{X^{s,-1/2}_S}+ \left\|\frac{\left<k\right>^s}{\left<\tau - \alpha k^2\right>}  \tilde f\left(k, \tau\right)\right\|_{l^2_k L^1_{\tau}}, \\
\left\|f\right\|_{Y_W^{s}} &= \left\|f\right\|_{X^{s,1/2}_W}+\left\|\left<k\right>^s  \tilde f\left(k, \tau\right)\right\|_{l^2_k L^1_{\tau}}, \\
\left\|f\right\|_{Z_W^{s}} &= \left\|f\right\|_{X^{s,-1/2}_W}+\left\|\frac{\left<k\right>^s}{\left<\left|\tau\right| - \beta \left|k\right|\right>}  \tilde f\left(k, \tau\right)\right\|_{l^2_k L^1_{\tau}}.
\end{align*}
we give some embeddings for the $Y$ and $Z$ spaces
\begin{align}
Y^s_{S,W} \subseteq &C_t H^s_x \subseteq L^{\infty}_t H^s_x, \label{eq:Y embedding} \\
L^2_t H^s &\subseteq Z^s_{S,W}. \notag
\end{align}
in a compact time interval $\left[0,T\right]$ by the H\"older inequality.
To simplify notations, $\mathcal{Y}$ spaces would be
\begin{equation*}
\left\|\left(u,v\right)\right\|_{\mathcal{Y}} = \left\|u\right\|_{Y^{0}_S} + \left\|v\right\|_{Y_W^{-\frac{1}{2}}}.
\end{equation*}

For each dyadic number $N$, we denote the Littlewood-Paley projection by
\begin{equation}\label{eq:Fourier multiplier}
\begin{split}
&\widehat{P_Nu}\left(k\right) := 1_{N \leq \left|k\right| <2N}\left(k\right) \widehat{u}_k, \\
&\widehat{P_{\leq N} u}\left(k \right) := 1_{\left|k\right| \leq N}\left(k\right) \widehat{u}_k, \\
&\widehat{P_{\geq N} u}\left(k \right) := 1_{\left|k\right| \geq N}\left(k\right) \widehat{u}_k,
\end{split}
\end{equation}
where $1_{\Omega}$ is a characteristic function on $\Omega$. 

We call a connected open set a domain. For an nonempty domain $\mathcal{O} \subset \mathcal{H}$ and $n \geq 1$, we denote
\begin{equation*}
\begin{aligned}
\mathcal{O}_N &= \mathcal{O} \cap P_{\leq N} \mathcal{H} \left(= P_{\leq N} L_x^2 \times P_{\leq N} H_x^{-\frac{1}{2}} \times P_{\leq N} H^{-\frac{3}{2}}_x\right),\\
\mathcal{O}^N &= \mathcal{O} \cap \left(1-P_{\leq N}\right) \mathcal{H},
\end{aligned}
\end{equation*}
and observe that
\begin{equation*}
\partial \mathcal{O}_N \subset \partial \mathcal{O} \cap P_{\leq N} \mathcal{H}.
\end{equation*}

For $x, y \in \mathbb{R}_+$, $x \lesssim y$ denotes $x \le Cy$ for some $C >0$ and $x \sim y$ means $x \lesssim y$ and $y \lesssim x$. Using this, we denote $f = O(g)$ by $f \lesssim g$ for positive real-valued functions $f$ and $g$. Moreover, $x \ll y$ denotes $x \le cy$ for small positive constant $c$. Let $a_1,a_2,a_3 \in \mathbb{R}$ and the quantities $a_{max} \ge a_{med} \ge a_{min}$ can be defined to be the maximum, median and minimum values of $a_1,a_2,a_3$, respectively.

\section{Symplectic capacity for Hilbert space}
We begin with the definition of the symplectic Hilbert space $\mathcal{H}$. Let $\omega$ be a symplectic form $\mathcal{H}$ as follows,
\begin{equation}\label{eq:symplectic form}
\omega\left(\left(u,v,w\right),\left(\acute u, \acute v, \acute w\right)\right) = \omega_i\left(u,\acute u\right) + \omega_{-1/2}\left(v,\acute v\right) + \omega_{-3/2}\left(w,\acute w\right)
\end{equation}
where Hello there,
When is its update?
Is it LIve? why dic
puaed? don't have
mayu

\begin{equation*}\label{eq:symplectic form for eqns}
\omega_i \left(f,g\right) =  \ON{Im} \int \overline{f}  {g}  dx, \quad \omega_{-1/2} \left(f,g\right) = \int f \partial_x^{-1} g dx, ~ \text{and} ~ \omega_{-3/2} \left(f,g\right) = \int f \partial_x^{-3} g dx.
\end{equation*}
Let $J$ be an almost complex structure on $\mathcal{H}$ which is compatible with the Hilbert space inner product $<\cdot, \cdot>$. In other words, a bounded self adjoint operator with $J^2 = -I$ such that $\omega\left({\bf u},{\bf v}\right) = <{\bf u}, J{\bf v}>$ for all ${\bf u}, {\bf v} \in \mathcal{H}$.
One easily checks that the Zakharov system can be written in the form
\begin{equation}\label{eq: Sym Hamil eq}
\partial_t {\bf u}\left(t\right) :=\dot {\bf u}\left(t\right) = J \nabla H\left[{\bf u}\left(t\right)\right]
\end{equation}
where ${\bf u} \in \mathcal{H}$. The notation $\nabla$ in \eqref{eq: Sym Hamil eq} denotes the usual gradient with respect to the Hilbert space inner product. 
Hence, we have
\begin{equation}\label{eq: grad analogy}
\begin{aligned}
\left<
 \begin{bmatrix}
            v_1\\
           v_2 \\
           v_3
 \end{bmatrix} 
,
 \begin{bmatrix}
            \nabla H\left[u\left(t\right), \cdot, \cdot\right]\\
           \nabla  H\left[\cdot, n\left(t\right), \cdot\right] \\
           \nabla  H\left[\cdot, \cdot, \dot n \left(t\right)\right]
 \end{bmatrix} 
\right>
& \equiv
\left<
 \begin{bmatrix}
           1 \\
           1 \\
          1
 \end{bmatrix} 
,
 \begin{bmatrix}
           d H\left[u\left(t\right), \cdot, \cdot\right] \left(v_1, \cdot, \cdot\right) \\
           d H\left[\cdot, n\left(t\right), \cdot\right] \left(\cdot,v_2,\cdot\right) \\
          d H\left[\cdot, \cdot, \dot n\left(t\right)\right] \left(\cdot,\cdot,v_3\right) 
 \end{bmatrix} 
 \right>.
 \\
 & \equiv
\left<
 \begin{bmatrix}
           1 \\
           1 \\
          1
 \end{bmatrix} 
,
  \begin{bmatrix}
           \nabla_{ \varepsilon_1} H \left[{ u }+{ \varepsilon_1 v}, \cdot, \cdot\right] \\
           \cdot, \nabla_{ \varepsilon_2} H \left[{ n }+{ \varepsilon_2 v}, \cdot\right] \\
           \cdot, \cdot, \nabla_{ \varepsilon_3} H \left[{\dot n }+{ \varepsilon_3 v}\right]
 \end{bmatrix} 
 \right>.
\end{aligned}
\end{equation}
\begin{definition}
Consider a pair $\left(\mathcal{H}, \omega\right)$, where $\omega$ is a symplectic form \eqref{eq:symplectic form} on the Hilbert space $\mathcal{H}\left(= L_x^2\left(\mathbb{T}\right) \times H_x^{-\frac{1}{2}}\left(\mathbb{T}\right) \times H^{-\frac{3}{2}}_x\left(\mathbb{T}\right)\right)$. We say that the pair $\left(\mathcal{H}, \omega\right)$ is the symplectic phase space for the Zakharov system. 
\end{definition}
One easily check that an equivalent way to write the Zakharov system corresponding to the Hamiltonian $\ref{eq:Hamiltonian}$ in $\left(\mathcal{H}, \omega\right)$ is
  \begin{align}
\partial_t   
 \begin{bmatrix}
            u \\
           n \\
           \dot n
 \end{bmatrix} 
=
 \begin{bmatrix}
            \nabla_{\omega_i}  H\left[u\left(t\right), \cdot, \cdot\right]\\
           \nabla_{\omega_{-1/2}}  H\left[\cdot, n\left(t\right), \cdot\right] \\
           \nabla_{\omega_{-3/2}}  H\left[\cdot, \cdot, \dot n \left(t\right)\right]
 \end{bmatrix} 
  \end{align}
where the symplectic gradient is defined an analogy with \eqref{eq: grad analogy},
\begin{equation*}
\begin{aligned}
\omega 
\left(
 \begin{bmatrix}
            v_1\\
           v_2 \\
           v_3
 \end{bmatrix} 
,
 \begin{bmatrix}
            \nabla_{\omega_i}  H\left[u\left(t\right), \cdot, \cdot\right]\\
           \nabla_{\omega_{-1/2}}  H\left[\cdot, n\left(t\right), \cdot\right] \\
           \nabla_{\omega_{-3/2}}  H\left[\cdot, \cdot, \dot n \left(t\right)\right]
 \end{bmatrix} 
\right)
=
\left<
 \begin{bmatrix}
           1 \\
           1 \\
          1
 \end{bmatrix} 
,
 \begin{bmatrix}
           d H\left[u\left(t\right), \cdot, \cdot\right] \left(v_1, \cdot, \cdot\right) \\
           d H\left[\cdot, n\left(t\right), \cdot\right] \left(\cdot,v_2,\cdot\right) \\
          d H\left[\cdot, \cdot, \dot n\left(t\right)\right] \left(\cdot,\cdot,v_3\right) 
 \end{bmatrix} 
 \right>.
\end{aligned}
\end{equation*}
Therefore, we can consider $L_x^2\left(\mathbb{T}\right) \times H_x^{-\frac{1}{2}}\left(\mathbb{T}\right) \times H^{-\frac{3}{2}}_x\left(\mathbb{T}\right)$ as the phase space. In the phase space, we can consider an invariance of the symplectic capacity by the Zakharov solution flow.

In the following, we  introduce the infinite dimensional symplectic capacity which was introduced by Kuksin \cite{Kuksin:1995ue}. The symplectic capacity was first discovered by Ekeland and Hofer \cite{Ekeland:1989dh,Ekeland:1990is} in $\mathbb{R}^{2n}$ and was developed by Hofer and Zehnder \cite{Hofer:2011vo}. Specifically, it is a symplectic invariant and the proof of existence is based on the variational principle.
\begin{definition}[Symplectic capacity]\label{defn: symplectic capacity}
A symplectic capacity on the phase space $\left(\mathcal{H},\omega\right)$ with respect to \ref{eq:Zakharov} is a function $\c\left( \cdot \right)$ defined open subset $\mathcal{O} \subset \mathcal{H}$ which takes values in $\left[0, \infty\right]$ and has the following properties:
\begin{enumerate}[{\rm i)}]
\item Translational invariant:
\begin{equation*}
\c\left(\mathcal{O}\right) = \c\left(\mathcal{O} + \xi\right) \quad \text{for} ~ \xi \in \mathcal{H}.
\end{equation*}
\item Monotonicity: Let $\mathcal{O}_1$ and $\mathcal{O}_2$ be open sets in $\mathcal{H}$.
\begin{equation*}
\c\left(\mathcal{O}_1\right) \leq \c\left(\mathcal{O}_2\right) \quad \text{if} ~ \mathcal{O}_1 \subseteq \mathcal{O}_2.
\end{equation*}
\item 2-homogeneity: For $\tau \in \mathbb{R}$,
\begin{equation*}
\c\left(\tau \mathcal{O}\right) = \tau^2 \c \left(\mathcal{O}\right).
\end{equation*}
\item Nontriviality: For bounded nonempty set $\mathcal{O}$,
\begin{equation*}
0 < \c \left(\mathcal{O}\right) < \infty.
\end{equation*}
\item For an $r$-ball $B^{\infty}_r$ in $\mathcal{H}$ and a $k$-th cylinder $C^{\infty}_{k,r}$ which has radius $r$ in $\mathcal{H}$, 
\begin{equation*}
\c\left(B^{\infty}_r\right) = \c\left(C^{\infty}_{k,r}\right) = \pi r^2.
\end{equation*}
\end{enumerate}
\end{definition}
We point out two notable remarks. Combining ii), v) and the invariant of the symplectic capacity, we can get the nonsqueezing theorem for a Hamiltonian flow. Moreover, we note that the symplectic capacity does not determine a unique capacity function. We, thus, have many ways to construct capacity functions, but we follow \cite{Kuksin:1995ue}. To construct a capacity function, we first introduce few definitions.
\begin{definition}[Admissible function]
Let $\mathcal{O}$ be a simply connected open set in the phase space $\mathcal{H}$ (we call it domain in the sequel). Assume that a function $f$ is a smooth function in $\mathcal{O}$, and $m>0$. The function $f$ is called $m$-admissible if
\begin{enumerate}[{\rm i)}]
\item $0 \leq f \leq m$ in $\mathcal{O}$,
\item $f \equiv 0$ in a nonempty subdomain of $\mathcal{O}$, 
\item $f \big|_{\partial \mathcal{O}} \equiv m$,
\item The set $\left\{f<m\right\}$ is bounded, and the distance from this set to $\partial \mathcal{O}$ is strictly positive, i.e., $\d_{\mathcal{O}}\left(f\right) (:= \d\left(\left\{f<m\right\}, \partial \mathcal{O}\right)) > 0$.
\end{enumerate}
\end{definition}
For each $m$-admissible function, we denote
\begin{equation*}
\sp f  = \left\{u: 0 < f\left(u\right) < m\right\},
\end{equation*}
and so we have
\begin{align}
\d\left(f^{-1}\left(0\right), \partial \mathcal{O} \right) \geq \d\left(f\right), \label{eq: boundary est. 1}\\
\d\left(\sp f, \partial \mathcal{O}\right) \geq \d\left(f\right). \label{eq: boundary est. 2}
\end{align}
Let $N \geq 1$ be integer. By using the Fourier basis and restricting the form $\omega$, we see that the truncated phase space $\left(P_{\leq N} \mathcal{H},  \omega\right)$ is a $3\times 2N $-dimensional real symplectic space, and so is symplectomorphic to the standard phase space $\left(\mathbb{R}^{2N \times 3}, \omega_0\right)$ by Darboux theorem.
\begin{definition}[Fast function]
Let $f_N = f \big|_{\mathcal{O}_N}$ be a frequency truncated function. We consider the corresponding symplectic Hamiltonian vector field $V_{f_N}$ in $\mathcal{O}_N$. In other words, for $u, v \in \mathcal{O}_N$,
\begin{equation*}
\omega\left({\bf v}, V_{f_N}\left({\bf u}\right)\right) = df_N\left({\bf u}\right)\left({\bf v}\right) .
\end{equation*}
We call a periodic trajectory of $V_{f_N}$ a `fast trajectory' if it does not pass through a stationary point and the period $T \leq 1$. Furthermore, a $m$-admissible function $f$ is called `fast' if there exists $N_0 = N_0\left(f\right)$ such that for all $N \geq N_0$ the vector field $V_{f_N}$ has a fast trajectory.
\end{definition}
A fast $m$-admissible function has a notable property which comes from its definition, 
\begin{lemma}[Kuksin \cite{Kuksin:1995ue}]\label{lem: support of fast traj.}
All fast periodic trajectories in $V_{f_N}$ are contained in $\sp f \cap P_{\leq N} \mathcal{H}$
\end{lemma}
This lemma can be proved by the definition of the fast trajectory and the fact that  the derivative of $f$ is zero in the complementary set of $\sp f$. We next introduce a definition of the (infinite dimensional) symplectic capacity.
\begin{definition}[Infinite dimensional symplectic capacity, Kuksin \cite{Kuksin:1995ue}]\label{defn: Infn symplectic capacity}
For a nonempty domain $\mathcal{O} \in \mathcal{H}$, its symplectic capacity $\c\left(\mathcal{O}\right)$ equals
\begin{equation*}
\inf \left\{m_* : \text{\rm each $m$-admissible function with $m>m_*$ is fast}\right\}
\end{equation*}
\end{definition}
From \cite{Kuksin:1995ue}, we already have that the symplectic capacity $\c\left(\cdot\right)$ satisfies Definition \ref{defn: symplectic capacity}. Note that the capacity $\c\left(\mathcal{O}\right)$ depends on the consecutive subsets $P_1 \mathcal{H} \subset P_2 \mathcal{H} \subset \cdots$ of the space $\mathcal{H}$.

\section{Basic estimates}\label{sec:estimates}
In this section, we show estimates to prove Theorem Theorem \ref{thm: main thm} and \ref{thm:GWP}.  First of all, we recall the lemma regarding the linear estimates. 
\begin{lemma}\label{lem:linear est.} Let $\Psi\left(t\right) \in C^{\infty}_0\left(\mathbb{R}\right)$ such that $\Psi\left(t\right) =1$ on $\left[-1,1\right]$ and $\Psi\left(t\right) =0$ outside of $\left[-2,2\right]$. We have
\begin{align*}
\left\| \Psi \left(\frac{t}{T}\right)U\left(t\right) u_0\right\|_{Y_S^{0}} &\lesssim \left\|u_0\right\|_{L^2_x},\\
\left\|\Psi\left(\frac{t}{T}\right) \int_0^t U\left(t-s\right) F\left(s\right) ds\right\|_{Y_S^{0}} &\lesssim  \left\|F\left(t\right) \right\|_{Z_S^{0}}, \\
\left\| \Psi \left(\frac{t}{T}\right)\partial_tV\left(t\right) n_0\right\|_{Y_W^{-1/2}} &\lesssim \left\|n_0\right\|_{H^{-\frac{1}{2}}_x} ,\\
\left\| \Psi \left(\frac{t}{T}\right)V\left(t\right) n_1\right\|_{Y_W^{-1/2}} &\lesssim  \left\|n_1\right\|_{H^{-\frac{3}{2}}_x}, \\
\left\|\Psi\left(\frac{t}{T}\right) \int_0^t V\left(t-s\right) \partial_xF\left(s\right) ds\right\|_{Y_W^{-1/2}} &\lesssim \left\|F\left(t\right) \right\|_{Z_W^{-1/2}},
\end{align*}
where $U\left(t\right) = e^{i\alpha t \partial^2_x}$ and $V\left(t\right) = \frac{\sin \left(\beta t \left(- \partial_x^2\right)^{1/2}\right)}{\beta \left(- \partial_x^2\right)^{1/2}}$. 
\end{lemma}
Lemma \ref{lem:linear est.} will be used for the contraction mapping principle, and so it is well-known estimates. For that reason, we omit the proof of Lemma \ref{lem:linear est.}, but it can be found in \cite{Bourgain:1993hz,Bourgain:1993cl,Bourgain:1994ej,Colliander:2008cq,KENIG:1996tq,Kenig:1996bu}. \\
\\
From now on, we discuss bilinear estimates with respect to the Schr\"odinger and the wave flow. Define the resonant set and nonresonant set as follows,
\begin{equation*}
\begin{aligned}
N_{\mathcal{R}} &= \left\{\left(k_0,k_1,k_2\right):k_0 \sim k_1 \sim k_2\right\} \\
N_{\mathcal{NR}} &= \left(N_{\mathcal{R}}\right)^C.
\end{aligned}
\end{equation*}

\begin{proposition}\label{prop:resonant bi est}
Let $N_i$ be dyadic numbers, $N_i \sim \left|k_i\right|$ and $0< T < 1$. Assume that $\frac{\beta}{\alpha}$ is not an integer,
\begin{equation}\label{eq:1st resonant bilinear est.}
\left\|P_{N_0}\left(P_{N_1} u P_{N_2} v\right)\right\|_{Z^{0}_S\left(\left[0,T\right] \times \mathbb{T}\right)} \lesssim T^{\gamma} \left\|u\right\|_{Y^{0}_S}  \left\|v\right\|_{Y^{-1/2}_W},
\end{equation}
and
\begin{equation}\label{eq:2nd resonant bilinear est.}
\left\|P_{N_0}\partial_x\left( P_{N_1} u \overline{P_{N_2} v}\right)\right\|_{Z^{-1/2}_W\left(\left[0,T\right] \times \mathbb{T}\right)} \lesssim T^{\gamma} \left\|u\right\|_{Y^{0}_S}  \left\|v\right\|_{Y^{0}_S}.
\end{equation}
for 
sufficiently small $\gamma > 0$.
\end{proposition}

\begin{proposition}\label{prop:bilinear est.}
Let $N_i$ be dyadic numbers and $N_i \sim \left|k_i\right|$. Assume that $\frac{\beta}{\alpha}$ is not an integer,
\begin{equation}\label{eq:1st bilinear est.}
\left\|P_{N_0}\left(P_{N_1} u P_{N_2} v\right)\right\|_{Z^{0}_S} \lesssim N^{-\delta}_{max} \left\|u\right\|_{Y^{0}_S}  \left\|v\right\|_{Y^{-1/2}_W},
\end{equation}
and
\begin{equation}\label{eq:2nd bilinear est.}
\left\|P_{N_0}\partial_x\left( P_{N_1} u \overline{P_{N_2} v}\right)\right\|_{Z^{-1/2}_W} \lesssim N^{-\delta}_{max} \left\|u\right\|_{Y^{0}_S}  \left\|v\right\|_{Y^{0}_S}.
\end{equation}
for $\left(k_0,k_1,k_2\right) \subset N_{\mathcal{NR}}$ and sufficiently small $\delta > 0$.
\end{proposition}

Takaoka \cite{Takaoka:1999uw} proved these type estimates (without dyadic decompositions) using the arguments in \cite{KENIG:1996tq, KENIG:1993ts}, but we have to show slightly stronger estimates for the approximation step in the proof of Theorem \ref{thm: main thm}. In addition, we prove the propositions using a little simpler argument in \cite{Tao:2001tu} than in \cite{Takaoka:1999uw,KENIG:1996tq, KENIG:1993ts}. Roughly, Proposition \ref{prop:bilinear est.} can be deduced by observation of the intersection of hyperspace $\tau = h\left(k\right)$ (a function $h$ be chosen by each equation).  We first have the following algebraic results for observation of resonant relations.

\begin{lemma}
Let ${\rm sgn}(x)$ be the sign function. Then \\
{\rm i)} Let $\tau_0 = \tau_1 + \tau_2$, $k_0  = k_1 + k_2 \ne 0$,
\begin{equation}\label{eq:1st resonant est.}
 \max \left\{ \left|\tau_0 - \alpha k_0^2\right|, \left|\tau_1 - \alpha k_1^2\right|, \left|\left|\tau_2\right| - \beta\left| k_2\right|\right|\right\} \gtrsim \left|\alpha\right| \left|k_2\right| \left|k_0+k_1 - \frac{\beta}{\alpha} S_1\right|,
\end{equation}
for $S_1 = {\rm sgn}\left(\tau_2 k_2\right)$.\\
{\rm ii)} Let $\tau_0 = \tau_1 - \tau_2$, $k_0  = k_1 - k_2 \not = 0$,
\begin{equation}\label{eq:2nd resonant est.}
 \max \left\{\left|\left|\tau_0\right| - \beta\left| k_0\right|\right|, \left|\tau_1 - \alpha k_1^2\right|, \left|\tau_2 - \alpha k_2^2\right|\right\} \gtrsim \left|\alpha\right| \left|k_0\right| \left|k_1 + k_2 - \frac{\beta}{\alpha} S_2\right|,
\end{equation}
for $S_2 = {\rm sgn}\left(\tau_0 k_0\right)$.
\end{lemma}
\begin{proof} We first prove i). From the support of time and spatial frequencies, 
\begin{align*}
\tau_0 - \alpha k_0^2  - \tau_1 + \alpha k_1^2 - \tau_2 + \beta S_1 k_2  &= \tau_0 - \tau_1 - \tau_2 + \alpha\left(k_1^2 - k_0^2 - \frac{\beta}{\alpha} S_1 k_2\right) \\
&= -\alpha k_2 \left( k_0 +k_1 - \frac{\beta}{\alpha}S_1\right).
\end{align*}
Similarly,
\begin{align*}
\tau_0 - \beta S_2 k_0 - \tau_1 + \alpha k_1^2 + \tau_2 - \alpha k_2^2  &= \tau_0 - \tau_1 + \tau_2 + \alpha\left(k_1^2 - k_2^2 - \frac{\beta}{\alpha} S_2 k_0\right) \\
&= \alpha k_0 \left( k_1 +k_2 - \frac{\beta}{\alpha}S_1\right).
\end{align*}
Therefore, we are done.
\end{proof}

We can obtain Proposition \ref{prop:resonant bi est} by the similar argument of Proposition \ref{prop:bilinear est.}, so we first prove the latter.

\begin{proof}[Proof of Proposition \ref{prop:bilinear est.}]
Let $L_i := \left<\tau_i - \alpha k_i^2\right>$ and $M_i = \left<\left|\tau_i\right| - \beta \left|k_i\right|\right>$. To prove (\ref{eq:1st bilinear est.}) and (\ref{eq:2nd bilinear est.}), we are going to show the $X^{s,b}$ part and $l^2_{k}L^1_{\tau}$ part respectively. However, essential ideas are similar.
From \cite{Bourgain:1993hz}, we have the Strichartz estimate for the Schr\"odinger equation, 
\begin{equation}\label{eq:Strichartz}
X_S^{0,3/8} \subset L^4_{t,x},
\end{equation}
which will be used many times in the following calculation. We first consider the $X^{s,b}$ part of (\ref{eq:1st bilinear est.}). The left hand side of (\ref{eq:1st bilinear est.}) can be rewritten by the Plancherel theorem, so our claim is 
\begin{equation}\label{eq:claim of the 1st bilinear est.}
\left\|\sum_{\substack{k_0=k_1+k_2 \\ (k_0,k_1,k_2) \in N_{\mathcal{NR}}}}\int_{\tau_0 = \tau_1 + \tau_2} \frac{\left|k_2\right|^{1/2}}{L_0^{1/2}L_1^{1/2}M_2^{1/2}}\tilde u\tilde v d\tau_1\right\|_{l^2_{k_0}L^2_{\tau_0}} \lesssim N_{max}^{-\delta}\left\|u\right\|_{L^2_{t,x}}\left\|v\right\|_{L^2_{t,x}}.
\end{equation}

\begin{enumerate}[i)]
\item $\max\left\{L_0, L_1,M_2\right\} = M_2$ \\
By (\ref{eq:1st resonant est.}), the left hand side of (\ref{eq:claim of the 1st bilinear est.}) is bounded by 
\begin{equation*}
\left\|\sum_{\substack{k_0=k_1+k_2 \\ (k_0,k_1,k_2) \in N_{\mathcal{NR}}}}\int_{\tau_0 = \tau_1 + \tau_2} \frac{\left|k_2\right|^{1/2}}{L_0^{1/2}L_1^{1/2}\left|\alpha\right|^{1/2}\left|k_2\right|^{1/2}\left|k_0+k_1- \frac{\beta}{\alpha}S_1\right|^{1/2}}\tilde u\tilde v d\tau_1\right\|_{l^2_{k_0}L^2_{\tau_0}}.
\end{equation*}
Hence, it is reduced to show that
\begin{equation*}\label{eq: nonreso goal}
\left\|\sum_{\substack{k_0=k_1+k_2 \\ (k_0,k_1,k_2) \in N_{\mathcal{NR}}}}\int_{\tau_0 = \tau_1 + \tau_2} \frac{N_{max}^{\delta}}{L_0^{1/2}L_1^{1/2}\left|k_0+k_1- \frac{\beta}{\alpha}S_1\right|^{1/2}}\tilde u\tilde v d\tau_1 \right\|_{l^2_{k_0}L^2_{\tau_0}} \lesssim \left\|u\right\|_{L^2_{t,x}}\left\|v\right\|_{L^2_{t,x}}.
\end{equation*}
Since $k_0 = k_1 + k_2$ and $(k_0,k_1,k_2) \subset N_{\mathcal{NR}}$, we have $k_{max}$ and $k_{med}$ such that $N_{max} \sim k_{max} \sim k_{med}$, and they are same sign. Thus, we have
\begin{equation*}
\frac{N_{max}^{\delta}}{\left|k_0+k_1- \frac{\beta}{\alpha}S_1\right|^{1/2}} \lesssim 1
\end{equation*}
for sufficiently small $\delta$. By duality, it is enough to show that
\begin{equation*}
\left|\iint_{\mathbb{R} \times \mathbb{T}} u_0u_1u_2 dx dt\right|\lesssim \left\|u_0\right\|_{X_{S}^{0,1/2}}\left\|u_1\right\|_{X_{{S}}^{0,1/2}}\left\|u_2\right\|_{L^2_{t,x}},
\end{equation*}
and then by the H\"older inequality and (\ref{eq:Strichartz}), we have
\begin{align*}
\left|\iint_{\mathbb{R} \times \mathbb{T}} u_0u_1u_2 dx dt\right| &\leq \left\|u_0\right\|_{L^4_{t,x}}\left\|u_1\right\|_{L^4_{t,x}}\left\|u_2\right\|_{L^2_{t,x}} \\
& \lesssim \left\|u_0\right\|_{X_{S}^{0,\frac{1}{2}}}\left\|u_1\right\|_{X_S^{0,\frac{1}{2}}}\left\|u_2\right\|_{L^2_{t,x}}.
\end{align*}

\item $\max\left\{L_0, L_1,M_2\right\} = L_i$\\
Without loss of generality, we may assume that $\max\left\{L_0, L_1,M_2\right\} = L_0$. From (\ref{eq:1st resonant est.}) again, the left hand side of (\ref{eq:claim of the 1st bilinear est.}) is bounded by 
\begin{equation*}
\left\|\sum_{\substack{k_0=k_1+k_2 \\ (k_0,k_1,k_2) \in N_{\mathcal{NR}}}}\int_{\tau_0 = \tau_1 + \tau_2} \frac{\left|k_2\right|^{1/2}}{\left|\alpha\right|^{1/2}\left|k_2\right|^{1/2}\left|k_0+k_1- \frac{\beta}{\alpha}S_1\right|^{1/2}L_1^{1/2}M_2^{1/2}}\tilde u\tilde v d\tau_1\right\|_{l^2_{k_0}L^2_{\tau_0}},
\end{equation*}
so the claim is 
\begin{equation}\label{eq:part of 1st AG}
\left\|\sum_{\substack{k_0=k_1+k_2 \\ (k_0,k_1,k_2) \in N_{\mathcal{NR}}}}\int_{\tau_0 = \tau_1 + \tau_2} \frac{N_{max}^{\delta}}{\left|k_0+k_1- \frac{\beta}{\alpha}S_1\right|^{1/2}L_1^{1/2}M_2^{1/2}}\tilde u\tilde v d\tau_1\right\|_{l^2_{k_0}L^2_{\tau_0}} \lesssim \left\|u\right\|_{L^2_{t,x}}\left\|v\right\|_{L^2_{t,x}}.
\end{equation}
Again, we have $k_0$ and $k_1$ in the former case, but we prove in a different way from i). There are several subcases. We first consider $\left|k_0\right| \sim \left|k_1\right|\gg\left|k_2\right| \sim N_{min}$. Since $k_0=k_1+k_2$, we have that $\left|k_0+k_1- \frac{\beta}{\alpha}S_1\right|$ is similar to $ N_{max}$. Thus, the left hand side of (\ref{eq:part of 1st AG}) is bounded by
\begin{equation*}
\left\|\sum_{\substack{k_0=k_1+k_2 \\ (k_0,k_1,k_2) \in N_{\mathcal{NR}}}}\int_{\tau_0 = \tau_1 + \tau_2} \frac{1}{\left|k_2\right|^{1/2-\delta}L_1^{1/2}M_2^{1/2}}\tilde u\tilde v d\tau_1\right\|_{l^2_{k_0}L^2_{\tau_0}}.
\end{equation*}
By duality, the claim equivalent to 
\begin{equation*}
\left|\iint_{\mathbb{R} \times \mathbb{T}} u_0u_1u_2 dx dt\right|\lesssim \left\|u_0\right\|_{L^2_{t,x}}\left\|u_1\right\|_{X_{S}^{0,1/2}}\left\|u_2\right\|_{X^{1/2-\delta, 1/2}_W}.
\end{equation*}
From the H\"older inequality, (\ref{eq:Strichartz}) and the Sobolev embedding with the time translation, 
\begin{align*}
\left|\iint_{\mathbb{R} \times \mathbb{T}} u_0u_1u_2 dx dt\right| &\leq \left\|u_0\right\|_{L^2_{t,x}}\left\|u_1\right\|_{L^4_{t,x}}\left\|u_2\right\|_{L^4_{t,x}} \\
& \lesssim \left\|u_0\right\|_{X_S^{0,0}}\left\|u_1\right\|_{X_{S}^{0,\frac{1}{2}}}\left\|u_2\right\|_{X^{1/2-\delta, \frac{1}{2}}_W},
\end{align*}
for sufficiently small $\delta$.
The remaining cases are $\left|k_1\right| \sim \left|k_2\right| \gg \left|k_0\right|$ or $\left|k_0\right| \sim \left|k_2\right| \gg \left|k_1\right|$. Then we have
\begin{equation}\label{eq: worst case1}
\frac{N^{\delta}_{max}}{\left|k_0+k_1-\frac{\beta}{\alpha}S_1\right|^{1/2}} \sim \frac{N^{\delta}_{max}}{\left|k_1\right|^{1/2}} \sim \frac{1}{\left|k_1\right|^{1/2-\delta}} \ll \frac{1}{\left|k_0\right|^{1/2-\delta}}
\end{equation}
or 
\begin{equation}\label{eq: worst case2}
\frac{N^{\delta}_{max}}{\left|k_0+k_1-\frac{\beta}{\alpha}S_1\right|^{1/2}} \sim \frac{N^{\delta}_{max}}{\left|k_0\right|^{1/2}} \sim \frac{1}{\left|k_0\right|^{1/2-\delta}}
\end{equation}
Thus, the left hand side of (\ref{eq:part of 1st AG}) is bounded by
\begin{equation}\label{eq: worst bound}
\left\|\sum_{\substack{k_0=k_1+k_2 \\ (k_0,k_1,k_2) \in N_{\mathcal{NR}}}}\int_{\tau_0 = \tau_1 + \tau_2} \frac{1}{\left|k_0\right|^{1/2-\delta}L_1^{1/2}M_2^{1/2}}\tilde u\tilde v d\tau_1\right\|_{l^2_{k_0}L^2_{\tau_0}}
\end{equation}
for both cases. By a similar calculation of the former case, we get the goal. More precisely, we need to show that
\begin{equation*}
\left|\iint_{\mathbb{R} \times \mathbb{T}} u_0u_1u_2 dx dt\right|\lesssim \left\|u_0\right\|_{L^2_{t}H^{1/2-\gamma}_x}\left\|u_1\right\|_{X_{S}^{0,1/2}}\left\|u_2\right\|_{X^{0, 1/2}_W}.
\end{equation*}
By the H\"older inequality, \eqref{eq:Strichartz}, and the Sobolev embedding,
\begin{align*}
\left|\iint_{\mathbb{R} \times \mathbb{T}} u_0u_1u_2 dx dt\right| &\leq \left\|u_0\right\|_{L^2_{t}L^4_x}\left\|u_1\right\|_{L^4_{t,x}}\left\|u_2\right\|_{L^4_{t}L^2_x} \\
& \lesssim \left\|u_0\right\|_{L^2_{t}H^{1/2-\gamma}_x}\left\|u_1\right\|_{X_{S}^{0,1/2}}\left\|u_2\right\|_{X^{0, 1/2}_W}.
\end{align*}
\end{enumerate}

We now consider $l^2_k L^1_{\tau}$ part,
\begin{equation}\label{eq:claim of the 1st bilinear est. l1 part}
\left\|\sum_{\substack{k_0=k_1+k_2 \\ (k_0,k_1,k_2) \in N_{\mathcal{NR}}}}\int_{\tau_0 = \tau_1 + \tau_2} \frac{\left|k_2\right|^{1/2}}{L_0L_1^{1/2}M_2^{1/2}}\tilde u\tilde v d\tau_1\right\|_{l^2_{k_0} L^1_{\tau_0}} \lesssim N^{-\delta}_{max} \left\|u\right\|_{Y^{0}_S}  \left\|v\right\|_{Y^{-1/2}_W}.
\end{equation}
First of all, from the Cauchy-Schwarz inequality and the fact that $L_0 = \left<\tau_0 - \alpha k_0^2\right>,$
\begin{align*}
&\left\|\sum_{\substack{k_0=k_1+k_2 \\ (k_0,k_1,k_2) \in N_{\mathcal{NR}}}}\int_{\tau_0 = \tau_1 + \tau_2} \frac{\left|k_2\right|^{1/2}}{L_0L_1^{1/2}M_2^{1/2}}\tilde u\tilde v d\tau_1\right\|_{l^2_{k_0} L^1_{\tau_0}} \\
&\lesssim \left\|\sum_{\substack{k_0=k_1+k_2 \\ (k_0,k_1,k_2) \in N_{\mathcal{NR}}}}\left\|\int_{\tau_0 = \tau_1 + \tau_2} \frac{\left|k_2\right|^{1/2}}{L_0^{15/32}L_1^{1/2}M_2^{1/2}}\tilde u\tilde v d\tau_1\right\|_{ L^2_{\tau_0}}\left|\int_{\mathbb{R}} \frac{1}{L_0^{17/16}} d\tau_0\right|^{1/2}\right\|_{l^2_{k_0}} \\
&\lesssim \left\|\sum_{\substack{k_0=k_1+k_2 \\ (k_0,k_1,k_2) \in N_{\mathcal{NR}}}}\int_{\tau_0 = \tau_1 + \tau_2} \frac{\left|k_2\right|^{1/2}}{L_0^{15/32}L_1^{1/2}M_2^{1/2}}\tilde u\tilde v d\tau_1\right\|_{l^2_{k_0} L^2_{\tau_0}}.
\end{align*}
We can use a similar calculation of $X^{s,b}$ part. More precisely, we can write our claim is,
\begin{equation}\label{eq:claim of the 1st bilinear est._Tail part}
\left\|\sum_{\substack{k_0=k_1+k_2 \\ (k_0,k_1,k_2) \in N_{\mathcal{NR}}}}\int_{\tau_0 = \tau_1 + \tau_2} \frac{\left|k_2\right|^{1/2}}{L_0^{15/32}L_1^{1/2}M_2^{1/2}}\tilde u\tilde v d\tau_1\right\|_{l^2_{k_0}L^2_{\tau_0}} \lesssim N_{max}^{-\delta}\left\|u\right\|_{L^2_{t,x}}\left\|v\right\|_{L^2_{t,x}}.
\end{equation}
The worst case is $\max\left\{L_0, L_1, M_2\right\} = L_0$ and $\left|k_2\right| \sim N_{max}$. By Lemma (\ref{eq:1st resonant est.}), the left hand side of (\ref{eq:claim of the 1st bilinear est._Tail part}) is bounded by 
\begin{equation*}
\left\|\sum_{\substack{k_0=k_1+k_2 \\ (k_0,k_1,k_2) \in N_{\mathcal{NR}}}}\int_{\tau_0 = \tau_1 + \tau_2} \frac{\left|k_2\right|^{1/2}}{\left|\alpha\right|^{15/32}\left|k_2\right|^{15/32}\left|k_0+k_1- \frac{\beta}{\alpha}S_1\right|^{15/32}L_1^{1/2}M_2^{1/2}}\tilde u\tilde v d\tau_1\right\|_{l^2_{k_0}L^2_{\tau_0}}.
\end{equation*}
We use \eqref{eq: worst case1} or \eqref{eq: worst case2} again, so it is bounded by \eqref{eq: worst bound}.From here, we can use the H\"older inequality, \ref{eq:Strichartz}, and the Sobolev embedding. Since $\frac{15}{32} > \frac{3}{8}$, the remaining cases can be proved by the similar process with $X^{s,b}$ part.

Next, we consider (\ref{eq:2nd bilinear est.}), but it can be proved by the similar calculation and using (\ref{eq:2nd resonant est.}) instead of (\ref{eq:1st resonant est.}). In other words, we can show that
\begin{equation}\label{eq:claim of the 2nd bilinear est.}
\begin{aligned}
\left\|\sum_{k_0=k_1-k_2}\int_{\tau_0 = \tau_1 - \tau_2} \frac{\left|k_0\right|^{1/2}}{M_0^{1/2}L_1^{1/2}L_2^{1/2}}\tilde u\tilde v d\tau_1\right\|_{l^2_{k_0}L^2_{\tau_0}} &\lesssim N_{max}^{-\delta}\left\|u\right\|_{L^2_{t,x}}\left\|v\right\|_{L^2_{t,x}}, \\
\left\|\sum_{k_0=k_1-k_2}\int_{\tau_0 = \tau_1 - \tau_2} \frac{\left|k_0\right|^{1/2}}{M_0L_1^{1/2}L_2^{1/2}}\tilde u\tilde v d\tau_1\right\|_{l^2_{k_0} L^1_{\tau_0}} &\lesssim N^{-\delta}_{max} \left\|u\right\|_{Y^{0}_S}  \left\|v\right\|_{Y^{0}_S}.
\end{aligned}
\end{equation}
As shown in (\ref{eq:claim of the 2nd bilinear est.}), it is similar to (\ref{eq:claim of the 1st bilinear est.}) and (\ref{eq:claim of the 1st bilinear est. l1 part}) except for indices. Hence, we can obtain (\ref{eq:2nd bilinear est.}) by (\ref{eq:2nd resonant est.}) and a similar argument for \eqref{eq:2nd bilinear est.}. 
\end{proof}

\begin{proof}[Proof of Proposition \ref{prop:resonant bi est}]
Note that Proposition \ref{prop:resonant bi est} has the time growth instead of the frequency decay. Hence, we need the following lemma.
\begin{lemma}[Lemma 2.11, \cite{Tao:2006tn}]\label{lem: time local Xsb}
Let $\eta$ be a Schwartz function in time and $I$ be a time interval. If $-1/2 < b' \leq b < 1/2$, then for any interval $\left[0,T\right] \subset I$ such that $0 < T < 1$, we have
\begin{equation*}
\left\|\eta\left(t/T\right)u\right\|_{X^{s,b'}\left(I \times \mathbb{T}\right)} \lesssim T^{b-b'} \left\|u\right\|_{X^{s,b}\left(I \times \mathbb{T}\right)}
\end{equation*}
\end{lemma}

\begin{proof}
By duality, we may assume that $0< b' \leq b < 1/2$. From the \emph{Christ-Kiselev lemma}, we suffices to show that
\begin{equation*}
\left\|\eta\left(t/T\right)u\right\|_{X_S^{s,b'}\left(\mathbb{R} \times \mathbb{T}\right)} \lesssim T^{b-b'} \left\|u\right\|_{X_S^{s,b}\left(\mathbb{R} \times \mathbb{T}\right)}.
\end{equation*}
From the following estimate,
\begin{equation*}
\left<\tau - \tau_0 - k^2\right>^b \lesssim \left<\tau_0\right>^{\left|b\right|} \left<\tau - k^2\right>^b,
\end{equation*}
we have
\begin{equation*}
\left\|e^{i t\tau_0} u\right\|_{X_S^{s,b}\left(\mathbb{R} \times \mathbb{T}\right)} \lesssim \left<\tau_0\right>^{\left|b\right|} \left\|u\right\|_{X_S^{s,b}\left(\mathbb{R} \times \mathbb{T}\right)}.
\end{equation*}
Since $\eta\left(t\right)$ is a Schwartz function, 
\begin{equation*}
\left\|\eta \left(t\right) u \right\|_{X_S^{s,b}\left(\mathbb{R} \times \mathbb{T}\right)} \lesssim \left(\int_{R} \left|\hat \eta \left(\tau_0\right)\right| \left<\tau_0\right>^{\left|b\right|} d \tau_0\right) \left\|u\right\|_{X_S^{s,b}\left(\mathbb{R} \times \mathbb{T}\right)} \lesssim \left\|u\right\|_{X_S^{s,b}\left(\mathbb{R} \times \mathbb{T}\right)}.
\end{equation*}
To simplify our argument, we first assume that $s=0$. Moreover, we may assume $b' =0$ by the duality with $b' = b$ case.
Thus, our claim is 
\begin{equation*}
\left\|\eta \left(t/T\right) u\right\|_{L_t^2 L_x^2 \left(\mathbb{R} \times \mathbb{T}\right)} \lesssim T^b \left\|u\right\|_{X_S^{0,b}\left(\mathbb{R} \times \mathbb{T}\right)}
\end{equation*}
for $0 < b < 1/2$. We now consider two cases separately, $\left<\tau - k^2\right> \geq 1/T$ and $\left<\tau - k^2\right> \leq 1/T$. 
In the former case, we have
\begin{equation*}
\left\|u\right\|_{X_S^{0,0}\left(\mathbb{R} \times \mathbb{T}\right)} \leq T^b \left\|u\right\|_{X_S^{0,b}\left(\mathbb{R} \times \mathbb{T}\right)}
\end{equation*}
by the fact that $\eta $ is a Schwartz function and $0< T< 1$. In the latter case, we have
\begin{equation*}
\begin{aligned}
\left\|\eta\left(t/T\right) u\right\|_{L_t^2 L_x^2} &\leq T^{1/2}\left\|\hat \eta\left(\tau\right) \right\|_{L_{\tau}^2} \left\|\mathcal{F}_t{u\left(t\right)}\left(k\right)\right\|_{L_t^{\infty} \ell_k^2} \\
&\lesssim T^{1/2}\left\|\hat \eta\left(\tau\right) \right\|_{L_{\tau}^2}\left\|\int_{\left<\tau - k^2\right> \leq 1/T} \left|\widetilde{u}\left(\tau, k\right) \right|d \tau \right\|_{\ell_k^2} \\
&\lesssim T^{1/2}\left\|\hat \eta\left(\tau\right) \right\|_{L_{\tau}^2} T^{b-1/2} \left\|\left(\int \left<\tau - k^2\right>^{2b} \left|\widetilde{u}\left(\tau, k\right)\right|^2 d \tau\right)^{1/2}\right\|_{\ell_k^2} \\
& = T^b \left\| u\right\|_{X_S^{s,b}\left(\mathbb{R} \times \mathbb{T}\right)}.
\end{aligned}
\end{equation*}
by the Plancherel and the triangle, the Cauchy-Schwarz, and the fact that $\eta $ is a Schwartz function.

Since the above argument does not depend $s$, we can get the same result in $s \in \mathbb{R}$.
\end{proof}

From Lemma \ref{lem: time local Xsb}, we have a modified embedding,
\begin{equation}\label{eq: new embedding}
\left\|\eta\left(\frac{t}{T}\right)u\right\|_{L^4_{t,x}} \lesssim \left\|\eta\left(\frac{t}{T}\right)u\right\|_{X_S^{0,\frac{3}{8}}} \lesssim T^{\gamma} \left\|u\right\|_{X_S^{0,3/8 + \gamma}}
\end{equation}
for sufficiently small $\gamma >0$. We now get Proposition \ref{prop:resonant bi est} by the similar argument in the proof of Proposition \ref{prop:bilinear est.} and \eqref{eq: new embedding} instead of \eqref{eq:Strichartz}.
\end{proof}

\begin{remark}
In fact, we can get the Proposition \ref{prop:resonant bi est} without the time growth. Hence from Proposition \ref{prop:resonant bi est} and summation with respect to each of dyadic frequency supports and the orthogonality of dyadic decomposition, we can obtain the full frequency estimate as follows,
\begin{equation*}\label{eq:full 1st bilinear est.}
\left\|uv\right\|_{Z^{0}_S} \lesssim  \left\|u\right\|_{Y^{0}_S}  \left\|v\right\|_{Y^{-1/2}_W},
\end{equation*}
and
\begin{equation}\label{eq:full 2nd bilinear est.}
\left\|\partial_x\left(  u \overline{v}\right)\right\|_{Z^{-1/2}_W} \lesssim  \left\|u\right\|_{Y^{0}_S}  \left\|v\right\|_{Y^{0}_S}.
\end{equation}
\end{remark}

\section{Global well-posedness}\label{sec:GWP}
In this section, we prove Theorem \ref{thm:GWP}, the global well-posedness for (\ref{eq:Zakharov}). It can be easily proved by combining the local well-posedness and the conservation law. More precisely, after splitting the time interval as a finite union of intervals (obviously, the length of each interval depends on \eqref{eq:mass conservation u}), we get the solution for each interval by the local well-posedness (see \cite{Takaoka:1999uw}) and glue each solution by using the mass conservation law of $u$. In particular, the nonlinear term of the wave part only consists of $u$, so the mass conservation of $u$ is sufficient to glue each solution\footnote{The details of this argument are in \cite{Colliander:2008cq}. In fact, Colliander et al. \cite{Colliander:2008cq} proved the global well-posedness of the Zakharov system on $\mathbb{R}$, but we can apply a similar argument to a torus.}.\\
\\
We briefly explain a sketch of the proof of Theorem \ref{thm:GWP}. It is sufficient to prove that
\begin{equation}\label{eq:GWP of wave part}
\sup_{\left|t\right|\le T}\left\|W\left(t\right)\left(u_0,n_0,n_1\right)\right\|_{H^{-1/2}_x} \lesssim C\left(T, \left\|u_0\right\|_{L^2_x}, \left\|n_0\right\|_{H_x^{-1/2}}, \left\|n_1\right\|_{H_x^{-3/2}}\right)
\end{equation}
and 
\begin{equation}\label{eq:GWP of Schrodinger part}
\sup_{\left|t\right|\le T}\left\|S\left(t\right)\left(u_0,n_0,n_1\right)\right\|_{ L^{{2}}_x} \lesssim C\left(T, \left\|u_0\right\|_{L^2_x}, \left\|n_0\right\|_{H_x^{-1/2}}, \left\|n_1\right\|_{H_x^{-3/2}}\right).
\end{equation}
for any $T>0$. To prove \eqref{eq:GWP of wave part} and \eqref{eq:GWP of Schrodinger part}, we define the norm as follows,
\begin{equation}\label{eq: new wave norm}
\left\|W\left(t\right) \left(u_0,n_0,n_1\right)\right\|_{\mathcal{W}} = \left\|(n,\partial_t n)\right\|_{\mathcal{W}} := \left(\left\|n\right\|_{H^{-1/2}}^2 + \left\|\partial_t n\right\|_{H^{-3/2}}^2\right)^{1/2}.
\end{equation}

Let $0<\alpha < 1$ be a constant, we can choose a sufficiently small time $T'$ such that 
\begin{align}
\left\|u\right\|_{Y^0_S\left[0,T'\right]} &\lesssim \left\|u_0\right\|_{L^2_x}, \label{eq:Bound of Y-norm}\\
\left\|u_0\right\|^2_{L^2_x} &\ll (T')^{-\alpha}\left\|(n_0,n_1)\right\|_{\mathcal{W}}, \label{eq:bound of L2-norm}
\end{align}
and 
\begin{equation}\label{eq: bound of wave data}
\left\|(n_0,n_1)\right\|_{\mathcal{W}} \lesssim (T')^{-1}
\end{equation}
by \eqref{eq:mass conservation u}. From \eqref{eq: new wave norm}, \eqref{eq:Y embedding}, \eqref{eq:Duhamel_wave}, the triangular inequality, Lemma \ref{lem:linear est.}, (\ref{eq:full 2nd bilinear est.}), and (\ref{eq:Bound of Y-norm}), we estimate
\begin{equation}\label{eq: bound of wave flow}
\begin{aligned}
\sup_{\left|t\right|\le T'}\left\|W\left(t\right)\left(u_0,n_0,n_1\right)\right\|_{\mathcal{W}} &\lesssim \left\|\Psi \left(\frac{t}{T'}\right)W\left(t\right)\left(u_0,n_0,n_1\right)\right\|_{Y^{-1/2}_W} \\
& \lesssim\left\|\Psi \left(\frac{t}{T'}\right)\partial_t V\left(t\right) n_0\right\|_{Y^{-1/2}_W} + \left\|\Psi \left(\frac{t}{T'}\right)V\left(t\right)n_1\right\|_{Y^{-1/2}_W} \\
&+ \left\|\Psi \left(\frac{t}{T'}\right)\beta^2 \int_0^t V\left(t-s\right) \left[\partial_x^2 \left(\left|u\right|^2\right)\right]\left(s\right)ds\right\|_{Y^{-1/2}_W} \\
&\lesssim \left\|n_0\right\|_{H^{-1/2}_x} + \left\|n_1\right\|_{H^{-3/2}_x} + \left\|\partial_x \left(\left|u\right|^2\right)\right\|_{Z^{-1/2}_W} \\
&\lesssim \left\|n_0\right\|_{H^{-1/2}_x} + \left\|n_1\right\|_{H^{-3/2}_x} + \left\|u\right\|^2_{Y_S^{0}} \\
&\lesssim \left\|(n_0,n_1)\right\|_{\mathcal{W}} + \left\|u_0\right\|^2_{L^2_x}.
\end{aligned}
\end{equation}
From (\ref{eq:bound of L2-norm}), we have
\begin{equation}\label{eq:LWP wave part}
\sup_{\left|t\right|\le T'}\left\|(W\left(t\right)\left(u_0,n_0,n_1\right)\right\|_{\mathcal{W}} \lesssim C\left(T', \left\|u_0\right\|_{L^2_x}, \left\|n_0\right\|_{H_x^{-1/2}}, \left\|n_1\right\|_{H_x^{-3/2}}\right)
\end{equation}
and by the similar calculation with \eqref{eq:Duhamel_Schrodinger},
\begin{equation}\label{eq:LWP schrodinger part}
\sup_{\left|t\right|\le T'}\left\|S\left(t\right)\left(u_0,n_0,n_1\right)\right\|_{ L^{{2}}_x} \lesssim C\left(T', \left\|u_0\right\|_{L^2_x}, \left\|n_0\right\|_{H_x^{-1/2}}, \left\|n_1\right\|_{H_x^{-3/2}}\right)
\end{equation}
for sufficiently small $T'$\footnote{In (\ref{eq:LWP schrodinger part}), the time $T'$ is the same as in (\ref{eq:LWP wave part}) because it is obtained by the local well-posedness, the fact that the nonlinear term of the Schr\"odinger part has $n\left(t,x\right)$, (\ref{eq:Bound of Y-norm}) and (\ref{eq:LWP wave part}).}. 

We now consider the gluing step. For any time $T$, we divide the total time interval $\left[-T,T\right]$ into time intervals such that each interval satisfies (\ref{eq:LWP wave part}) and (\ref{eq:LWP schrodinger part}). In the first such interval $\left[0, T'\right]$, we directly obtain (\ref{eq:GWP of wave part}) and (\ref{eq:GWP of Schrodinger part}), and in the next interval, we let $T'$ be the initial time and then obtain the claim by (\ref{eq:mass conservation u}). Hence, we can use the same iteration up to $\left\|W(\widetilde{T} )\left(u_0,n_0,n_1\right)\right\|_{\mathcal{W}} \gg \left\|u_0\right\|_{L^2_x}^2$. By taking this time as the initial time $\widetilde{T} =0$, we can repeat the entire procedure again. To reach the given time $T$, we need to show that time $\widetilde{T}$ is independent of $W\left(t\right)\left(u_0,n_0, n_1\right)$. From the final term in  \eqref{eq: bound of wave flow} and \eqref{eq:bound of L2-norm}, we can iterate $m$-times such that
\begin{equation*}
m \sim \frac{\left\|\left(n_0, n_1\right)\right\|_{\mathcal{W}}}{\left\|u_0\right\|^2_{L^2_x}},
\end{equation*}
and from \eqref{eq: bound of wave data}, we have
\begin{equation*}
\widetilde{T} = mT \lesssim \frac{1}{\left\|u_0\right\|^2_{L^2_x}}
\end{equation*}
which is independent of $W\left(t\right)\left(n_0,n_1\right)$. Therefore, we are done.

\section{Proof of Theorem \ref{thm: main thm}}
In this section, we prove Theorem \ref{thm: main thm}, the invariant of symplectic capacity with respect to the Zakharov flow. 
\subsection{Local approximation}
We introduce a new system as follows,
\begin{equation}\label{eq: modified truncated equation}
\left\{\begin{array}{ll}
i \partial_t \left(P_{\leq N}u\right) + \alpha \partial_x^2 \left(P_{\leq N}u\right) =P_{\le N}\left[\left(P_{\leq N}u\right)\left(P_{\leq N}n\right)\right], \\
\beta^{-2} \partial_t^2 \left(P_{\leq N}n\right) - \partial_x^2 \left(P_{\leq N}n\right) =P_{\le N} \left[\partial_x^2 \left(\left|P_{\leq N}u\right|^2\right)\right], \\
i \partial_t \left(\left(1-P_{\leq N}\right)u\right) + \alpha \partial_x^2 \left(\left(1-P_{\leq N}\right)u\right) =0,  \\
\beta^{-2} \partial_t^2 \left(\left(1-P_{\leq N}\right)n\right) - \partial_x^2 \left(\left(1-P_{\leq N}\right)n\right) =0
\end{array}\right.
\end{equation}
for the initial data $\left(u_0,n_0,n_1\right) \in \mathcal{H}$. Let ${Z}^N\left(t\right)$ be a solution flow with respect to \eqref{eq: modified truncated equation}, and its Hamiltonian is 
\begin{equation*}
{H}^N\left[u,n,\dot n\right] = \int_{\mathbb{T}} \left(\alpha \left|\partial_x u\right|^2 + \frac{\left| n\right|^2}{2} + \beta^2\frac{\left|i \partial_x^{-1} \dot  n\right|^2}{2} + P_{\leq N} n \left|P_{\leq N} u\right|^2 \right) dx.
\end{equation*}
Note that the new system \eqref{eq: modified truncated equation} has the same nonlinear operator in the low frequencies, and a linear operator in the high frequencies. Hence, the solution map ${Z}^N\left(t\right)$ is a smooth symplectomorphism with the symplectic form \eqref{eq:symplectic form} on the Hilbert  space $\mathcal{H}$, and the new system has the global well-posedness as well.

\begin{proposition}\label{prop:trunc of flow}
For a global-in-time $T > 0$ and any large integer $N$. The initial data $\left(u_0,n_0,n_1\right)$ is in $\mathcal{H}$. Assume that $Z\left(t\right)$ and ${Z}^N\left(t\right)$ be the Zakharov flow and the solution flow for \eqref{eq: modified truncated equation}, respectively. Then we have
\begin{equation*}\label{eq: claim of bound}
\sup_{\left|s\right|\leq t}\left\|\left(Z\left(s\right)\left(u_0,n_0,n_1\right)-{Z}^{N}\left(s\right)\left(u_0,n_0,n_1\right)\right)\right\|_{\mathcal{H}} \leq C\left(T, \left\|\left(u_0,n_0,n_1\right)\right\|_{\mathcal{H}}\right) N^{-\delta}
\end{equation*}
for a local-in-time $0<t \ll 1$ and $\delta >0 $.
\end{proposition}
\begin{proof}
We denote that ${\bf{z}}_0:=\left(u_0,n_0,n_1\right) $, ${\bf{z}}\left(t\right)= \left(u\left(t\right),n\left(t\right),\partial_t u\left(t\right)\right):=Z\left(t\right) {\bf{z}}_0 $ and ${\bf{z}}^N\left(t\right)= \left(u^N\left(t\right), n^N \left(t\right), \partial_t n^N\left(t\right)\right):= {Z}^N (t){\bf{z}}_0$. From the global well-posedness, there exists constant ${C}\left(T,\left\|{\bf{z}}_0\right\|_{\mathcal{H}}\right)$ such that
\begin{equation}\label{eq:small}
\left\|{\bf{z}}\left(t\right)\right\|_{\mathcal{Y}} + \left\|{\bf{z}}^N\left(t\right)\right\|_{\mathcal{Y}} \le {C}\left(T,\left\|{\bf{z}}_0\right\|_{\mathcal{H}}\right) := \mathcal{R}.
\end{equation}
We split the solution into two portions as follows,
\begin{equation*}
\begin{aligned}
{\bf{z}}\left(t\right) = {\bf{z}}_{lo} +{\bf{z}}_{hi} &:= P_{\leq N} {\bf{z}}\left(t\right) + \left(1-P_{\leq N}\right) {\bf{z}}\left(t\right)\\
&=P_{\leq N }u\left(t\right) + P_{\leq N }n\left(t\right)+ \left(1-P_{\leq N}\right) u\left(t\right)  + \left(1-P_{\leq N}\right) n\left(t\right)\\
&=:u_{lo}+n_{lo} +u_{hi}+n_{hi}.
\end{aligned}
\end{equation*}
By \eqref{eq:small}, we also have
\begin{equation*}\label{eq:split}
\left\|{\bf{z}}_{lo}\right\|_{\mathcal{Y}} \le \mathcal{R} ~ \text{and} ~\left\|{\bf{z}}_{hi}\right\|_{\mathcal{Y}} \le \mathcal{R}. 
\end{equation*}
Likewise, ${\bf{z}}^N$ is also split, and is bounded by $\mathcal{R}$ for each flow. Especially, $u^N_{hi}$ and $n^N_{hi}$ are linear flow by the definition of the new Hamiltonian system flow. By the structure of the wave part and \eqref{eq:Y embedding},
\begin{equation*}
\sup_{\left|s\right| \leq t} \left\|Z\left(t\right)\left(u_0,n_0,n_1\right)-{Z}^{N}\left(t\right)\left(u_0,n_0,n_1\right)\right\|_{\mathcal{H}} \lesssim \left\|Z\left(t\right)\left(u_0,n_0,n_1\right)-{Z}^{N}\left(t\right)\left(u_0,n_0,n_1\right)\right\|_{\mathcal{Y}}
\end{equation*}
The right hand side is bounded by
\begin{equation*}
\left\|\int_0^t U\left(t-s\right) \left[un - P_{\leq N}\left(u_{lo}n_{lo}\right)\right]\left(s\right) ds\right\|_{Y_S^0} + \left\|\int_0^t V\left(t-s\right) \partial_x^2\left[\left|u\right|^2 - P_{\leq N} \left|u_{lo}\right|^2\right]\left(s\right) ds\right\|_{Y_S^{-1/2}}
\end{equation*}
by the Duhamel's formula and the fact that the initial data is same. We first estimate the Schr\"odinger part. By Lemma \ref{lem:linear est.} and the Minkowski inequality, we have
\begin{equation}\label{eq: est of Sch part}
\begin{aligned}
&\left\|\int_0^t U\left(t-s\right) \left[un - P_{\leq N}\left(u_{lo}n_{lo}\right)\right] ds\right\|_{Y_S^0}  \lesssim \left\|un - P_{\leq N}\left(u_{lo}n_{lo}\right)\right\|_{Z^0_S}\\
&= \left\|un - P_{\leq N}\left(un\right)+P_{\leq N}\left(un\right)-P_{\leq N}\left(u_{lo}n\right)+P_{\leq N}\left(u_{lo}n\right)- P_{\leq N}\left(u_{lo} n_{lo}\right)\right\|_{Z_S^0} \\
&\leq \left\|\left(1- P_{\leq N}\right)\left(un\right)\right\|_{Z_S^0} +\left\|P_{\leq N}\left(\left(u-u_{lo}\right)n\right)\right\|_{Z_S^0}+\left\|P_{\leq N}\left(u_{lo}\left(n- n_{lo}\right)\right)\right\|_{Z_S^0}
\end{aligned}
\end{equation}
We apply \eqref{eq:1st bilinear est.} to the first term, and \eqref{eq:1st resonant bilinear est.} to the second term and the third term, thus the right hand side of \eqref{eq: est of Sch part} is bounded by
\begin{equation*}
\begin{aligned}
& N^{-\delta} \left\|u\right\|_{Y_S^0}\left\|n\right\|_{Y_W^{-1/2}} + t^{\gamma} \left\|n\right\|_{Y_W^{-1/2}}\left\|u-u_{lo}\right\|_{Y_S^0} +t^{\gamma} \left\|u_{lo}\right\|_{Y_S^{0}}\left\|n-n_{lo}\right\|_{Y_W^{-1/2}} \\	
=& N^{-\delta} \left\|u\right\|_{Y_S^0}\left\|n\right\|_{Y_W^{-1/2}} + t^{\gamma} \left(\left\|n\right\|_{Y_W^{-1/2}}\left\|u-u_{lo}\right\|_{Y_S^0} +\left\|u_{lo}\right\|_{Y_S^{0}}\left\|n-n_{lo}\right\|_{Y_W^{-1/2}}\right).
\end{aligned}
\end{equation*}
By the global well-posedness \eqref{eq:small}, we have the estimate for the Schr\"dinger part as follows,
\begin{equation*}
\left\|\int_0^t U\left(t-s\right) \left[un - P_{\leq N}\left(u_{lo}n_{lo}\right)\right] ds\right\|_{Y_S^0} \lesssim \mathcal{R}^2N^{-\delta}  + t^{\gamma}\mathcal{R}\left\| {\bf{z}}\left(t\right) - {\bf{z}}^N\left(t\right)\right\|_{\mathcal{Y}}.
\end{equation*}

By the similar calculation with \eqref{eq:2nd bilinear est.}, \eqref{eq:2nd resonant bilinear est.}, and the global well-posedness, the wave part is bounded as well. Indeed, 
\begin{equation*}
\begin{aligned}
\left\|\int_0^t V\left(t-s\right)\partial_x^2 \left[\left|u\right|^2 - P_{\leq N} \left|u_{lo}\right|^2\right]\left(s\right) ds\right\|_{Y_W^{-1/2}} &\lesssim N^{-\delta} \left\|u\right\|^2_{Y_S^0} + t^{\gamma} \left\|u\right\|_{Y_S^{0}}\left\|u-u_{lo}\right\|_{Y_S^0} \\
&\lesssim \mathcal{R}^2N^{-\delta}  + t^{\gamma}\mathcal{R}\left\| {\bf{z}}\left(t\right) - {\bf{z}}^N\left(t\right)\right\|_{\mathcal{Y}}.
\end{aligned}
\end{equation*}

Therefore, we have
\begin{equation*}
\begin{aligned}
&\left\|Z\left(t\right)\left(u_0,n_0,n_1\right)-{Z}^{N}\left(t\right)\left(u_0,n_0,n_1\right)\right\|_{\mathcal{Y}}  \\
& \lesssim \left\|\int_0^t U\left(t-s\right) \left[un - P_{\leq N}\left(u_{lo}n_{lo}\right)\right]\left(s\right) ds\right\|_{Y_S^0} + \left\|\int_0^t V\left(t-s\right)\partial_x^2 \left[\left|u\right|^2 - P_{\leq N} \left|u_{lo}\right|^2\right]\left(s\right) ds\right\|_{Y_W^{-1/2}}   \\
&\lesssim \mathcal{R}^2N^{-\delta}  + t^{\gamma}\mathcal{R}\left\|Z\left(t\right)\left(u_0,n_0,n_1\right)-{Z}^{N}\left(t\right)\left(u_0,n_0,n_1\right)\right\|_{\mathcal{Y}} .
\end{aligned}
\end{equation*}
Thus, choosing local-in-time $t$ such that $t<\left(\frac{1}{\mathcal{R}}\right)^{\frac{1}{\gamma}}$, we have
\begin{equation*}
\begin{aligned}
&\left\|Z\left(t\right)\left(u_0,n_0,n_1\right)-{Z}^{N}\left(t\right)\left(u_0,n_0,n_1\right)\right\|_{\mathcal{Y}} \\
 & \lesssim R N^{-\delta}  = C\left(T, \left\|\left(u_0,n_0,n_1\right)\right\|_{\mathcal{H}}\right)N^{-\delta}.
\end{aligned}
\end{equation*}
\end{proof}

\begin{remark}
The local-in-time $t$ in Proposition \ref{prop:trunc of flow} does not depend on frequency $N$. We thus conclude that the map $[Z-{Z}^N]\left(t\right)$ is regarded as  a small perturbation in a sufficiently short time interval.
\end{remark}

\subsection{Proof of symplectic invariant}

We separate the solution flow, and use an iteration argument. There exists a local time length $\tau\left(T, \left\|\left(u_0,n_0,n_1\right)\right\|_{\mathcal{H}}\right)>0$ such that the Zakharov flow satisfies Proposition \ref{prop:trunc of flow}. The global time interval $\left[0, T\right]$ is split to $\left[0=t_0, t_1\right] \cup \left[t_1, t_2\right] \cup \cdots \cup \left[t_{n-1}, t_n=T\right]$, and length of each interval is the constant $\tau\left(=\left|\tau_{i+1} - \tau_i\right|\right)$ that depends only on the implicit constant in Proposition \ref{prop:trunc of flow}. 
Let $\Omega_0$ be a initial domain which contains the initial data $u\left(x,0\right)$. Likewise, we denote that $\Omega_i\left(:= Z\left(t_i\right) \left(\Omega_0\right)\right)$ is a domain which has the solution $u\left(x,t_i\right)$. 
\\
{\noindent{\bf{FIRST STEP}}} (Local-time symplectic invariant)
\\
We first prove that  
\begin{equation}\label{eq: pre cap one side}
 \c \left(\Omega_1\right)= \c \left(Z\left(t_1\right) \left(\Omega_0 \right)\right)  \leq \c \left(\Omega_0\right).
\end{equation}
Let $f_1$ be a $m$-admissible function in $\Omega_1$ such that $m > \c\left( \Omega_0\right)$. From the Definition \ref{defn: Infn symplectic capacity}, it suffices to show that the function $f_1$ is a fast function in $\Omega_1$. Since the fact that the initial domain $\Omega_0$ is bounded and the Zakharov system has the global well-posedness, the domain $\Omega_1$ is a bounded domain as well. Denoting $\widetilde{\Omega}_0 = \Omega_0 \cap \left(Z\left(t_1\right)\right)^{-1} \left(\Omega_1\right)$, we have $\c\left(\widetilde{\Omega}_0\right) \leq \c \left(\Omega_0\right)$ by Definition \ref{defn: symplectic capacity}. Define $\varepsilon_1 = \d_{\Omega_1} \left(f_1\right)$, we can get a sufficiently large integer $N$ such that 
\begin{equation*}
\frac{\varepsilon_1}{2} > N_1^{- \delta}
\end{equation*}
where $\delta$ is the implicit constant in Proposition \ref{prop:trunc of flow}. The Zakharov flow is decomposed to 
\begin{equation*}
\begin{aligned}
Z\left(t_1\right) &= Z\left(t_1\right ) - Z^{N_1} \left(t_1\right) + Z^{N_1} \left(t_1\right)  \\
&= \left[I+ \left(Z\left(t_1\right) - {Z}^{N_1}\left(t_1\right)\right) \circ \left({Z}^{N_1}\left(t_1\right)\right)^{-1}\right] \circ {Z}^{N_1}\left(t_1\right) \\
&=: \left(I +Z_{\varepsilon_1}\left(t_1\right)\right)\circ {Z}^{N_1}\left(t_1\right),
\end{aligned}
\end{equation*}
where $I$ is an identity map from $\mathcal{H}$ to $\mathcal{H}$, and the map $Z^{N_1} \left(t_1\right)$ is the solution map for \eqref{eq:split}. Note that $I+ Z_{\varepsilon_1}\left(t_1\right)$ and $Z^{N_1}\left(t_1\right)$ are smooth symplectomorphisms. In the low frequencies, the solution map $Z^{N_1}\left(t_1\right)$ is composite operator with linear and nonlinear solution operators which is a finite dimensional symplectomorphism. In the high frequencies, the map are linear solution operator only, and they are isometries on the symplectic Hilbert space $\mathcal{H} \left(= L^2 \times H^{-1/2} \times H^{-3/2}\right)$. Hence, the classes of $m$-admissible functions are preserved by $Z^{N_1}\left(t_1\right)$. We thus show that 
\begin{equation}\label{eq: pre cap for approx op}
\c \left(\left(I + Z_{\varepsilon_1} \left(t_1\right)\right) (\acute{\Omega}_0 ) \right)  \leq \c \left(\acute{\Omega}_0 \right).
\end{equation}
where a domain $\acute{\Omega}_0 = Z^{N_1}\left(t_1\right)\left(\widetilde{\Omega}_0\right)$. 
By the decomposition of the Zakharov flow, we have 
\begin{equation*}
\left(I + Z_{\varepsilon_1} \left(t_1\right)\right) (\acute{\Omega}_0 )= \Omega_1.
\end{equation*}
Since an inverse operator $\left(Z^{N_1} \left(t_1\right)\right)^{-1}$ is also bounded, the operator $\left(Z\left(t_1\right) - Z^{N_1} \left(t_1\right)\right) \circ \left(Z^{N_1} \left(t_1\right)\right)^{-1}$ has an estimate
\begin{equation}\label{eq: approximation for trunc operator}
\left\|\left(Z\left(t_1\right) - Z^{N_1} \left(t_1\right)\right) \circ \left(Z^{N_1} \left(t_1\right)\right)^{-1} \right\|_{\acute\Omega_0 \to \Omega_1} \lesssim C\left(T,\acute \Omega_0 \right) N\left(\varepsilon_1\right)^{-\delta} =: N_1 \left(T, \acute \Omega_0, \varepsilon_1\right) ^{-\delta}
\end{equation}
for the constant $\delta >0$ by Proposition \ref{prop:trunc of flow}.\\
Let $V_{f_j^1}$ be a vector fields of the function $f_1$.  It suffices to show that the vector field $ V_{f_j^1}$ have a fast trajectory in the domain $\Omega_1$, for large integer $j$. The function $f_1$ is extended as $m$ outside $\Omega_1$, and provides an extended smooth function $g$ in $\mathcal{H}$. Moreover, let $h$ be a function which is restriction $g$ to $\acute \Omega_0$. Since the operator $Z_{{\varepsilon}_1}\left(t_1\right)$ has a estimate \eqref{eq: approximation for trunc operator}, the $\varepsilon_1$-neighborhood of $\partial {\Omega}_1$ is enclosed in the $\frac{ \varepsilon_1}{2}$-neighborhood of $\acute\Omega_0$, where $h \equiv m$. Furthermore, we have $h^{-1}\left(0\right)= f_1^{-1}\left(0\right) \subset \acute{\Omega}_0 \cap {\Omega}_1$ by \eqref{eq: boundary est. 1}. In other words, $\sp h$ is equal to $\sp f_1$. Hence, the function $h$ is an $m$-admissible function in $\acute {\Omega}_0$. Since $m>\c\left(\acute\Omega_0\right)$, the vector field $V_{h_j}$ has a fast trajectory in $\acute \Omega_0$ for all $j \gg 1$. By Lemma \ref{lem: support of fast traj.}, this trajectory lies in $\sp h$, which equals $\sp f$ by \eqref{eq: boundary est. 2}. Therefore, the vector field $V_{f_j}$ has a fast trajectory in ${\Omega}_1$ for all $j \gg 1$. That is, the function $f_1$ is fast in ${\Omega}_1$. \\
The opposite case can be shown by the same argument for the inverse operator. Therefore, we have
\begin{equation*}
\c \left(\Omega_1\right) = \c \left(Z\left(t_1\right) \Omega_0\right) = \c \left(\Omega_0\right)
\end{equation*}
for the local time $t_1$.
\\
{\noindent{\bf{SECOND STEP}}} (Iteration step)
\\
Fix a domain $\Omega_{i-1}$ for any time $t_{i}$, we can get an appropriate constant $\varepsilon_{i}$ which is depended on $\d_{\Omega_{i}}\left(f_{i}\right)$. Thus we have $N_{i}\left(T, \acute \Omega_{i-1}, \varepsilon_i\right)$, and so we show that the symplectic capacity is preserved for $\left[t_i, t_{i+1}\right]$ by the similar argument of the first step, since the constant $N_{i}\left(T, \acute \Omega_{i-1}, \varepsilon_{i}\right)$ is independent of the local time length $\tau$. Repeating the process, we have that the Zakharov flow preserves the symplectic capacity is its the phase space for the given global-time $T$.

\bibliography{Capacity_Za}
\bigskip
\end{document}